\documentclass[11pt]{amsart}
\usepackage[utf8]{inputenc}
\usepackage[margin=1in]{geometry}
\usepackage{hyperref}
\usepackage{amsfonts}
\usepackage{amsmath}
\usepackage{amsthm}
\usepackage{enumerate}
\usepackage{mathrsfs}
\usepackage{tikz}

\numberwithin{equation}{section}

\newcommand{\sgn}{\mathop{\mathrm{sgn}}}
\newcommand{\indentitem}{\setlength\itemindent{.25in}}
\newcommand{\D}{\mathscr{D}}
\newcommand{\R}{\mathbb{R}}
\newcommand{\dom}{\operatorname{dom}}

\theoremstyle{definition}
\newtheorem{defn}{Definition}
\newtheorem{lem}{Lemma}

\newtheorem{thm}{Theorem}

\newtheorem{prop}{Proposition}

\title{Fourier and Beyond: Invariance Properties of a Family of Integral Transforms}
\author{
Cameron L. Williams}
\address{
Department of Mathematics \\ 
University of Houston \\ 
Houston, TX 77204-3008, USA}
\author{
Bernhard G. Bodmann}
\address{
Department of Mathematics \\ 
University of Houston \\ 
Houston, TX 77204-3008, USA}
\author{
Donald J. Kouri}
\address{
Departments of Physics, Mathematics and Mechanical Engineering \\
University of Houston \\ 
Houston, TX 77204-3008, USA}

\begin{document}

\begin{abstract} 
The Fourier transform is typically seen as closely related to the additive group of real numbers, its characters and its Haar measure. In this paper, we propose an alternative viewpoint; the Fourier transform can be uniquely characterized by an intertwining relation with dilations and by having a Gaussian as an eigenfunction. 
This broadens the perspective to an entire family of Fourier-like transforms that are uniquely identified by the same dilation property and having Gaussian-like functions as eigenfunctions. We show that these transforms share many properties with the Fourier transform, particularly unitarity, periodicity and eigenvalues. We also establish short-time analogues of these transforms and show a reconstruction property and an orthogonality relation for the short-time transforms.
\end{abstract}

\maketitle

\section{Introduction}

The Fourier transform is central to many results in science, engineering and mathematics. There are at least three incarnations of the Fourier transform on $\R$: One can understand it as an integral transform that applies to Lebesgue integrable functions; with an appropriate normalization it can be viewed as a unitary map on the Hilbert space $L^2({\mathbb R})$; and it is also possible to define the Fourier transform of distributions by appealing to duality.

Through its connection with the group structure of the real numbers, the Fourier transform appears to leave little flexibility in its design. 
Part of the popularity of wavelets in harmonic analysis can be attributed to the variety of ways that scaling functions and associated wavelets,
building blocks for signal analysis, can be chosen for different purposes in
time-frequency analysis. A structural difference between Fourier and wavelet analysis is the use of dilations 
in the definition of the wavelet transform which are related to the affine group.

This paper follows a path to the Fourier transform that orients itself closely with the structure of wavelets, making the intertwining relationship
with dilations a defining property instead of the relationship between translations and modulations.
 If one had to choose a ``scaling function'' associated
with the Fourier transform, it would arguably be the Gaussian; a suitably chosen Gaussian is an eigenvector of the Fourier transform, and its translates and
modulations are related by analytic continuation. The Gaussian is also an uncertainty minimizer, a fact that has relevance for the short-time Fourier transform,  which extracts local frequency information by modulating with a moving window and subsequently applying the Fourier transform.

The Gaussian also plays the role of a low-pass filter in some applications; however, while it is localized in both time and frequency, it is not considered to be close to an ideal low-pass filter \cite{bodmann}. Particularly, the Gaussian decays rapidly to zero, whereas a nearly ideal low-pass filter should be nearly one over an interval of interest and decay rapidly to zero outside of it. On the other hand, because of the nearly discontinuous behavior of an approximately ideal filter, one would not expect its Fourier transform to be well-localized which is desirable from a numerical point of view. The simplest way in which to address this deficiency is to find an integral transform with an eigenfunction possessing simultaneous localization in both domains, something that seems impossible when considering the various manifestations of the uncertainty principle in harmonic analysis.

In this paper we address the question of whether the Fourier transform can be generalized in some way so that this can be achieved and a more general frequency and time-frequency analysis can be performed. The key tenet will be to develop transforms that have low-pass filters as eigenfunctions. To uncover a more general family of integral transforms, we start by considering some of the most fundamental  properties of the Fourier transform. The Fourier transform of any $f \in L^1(\mathbb R)$ is by our convention
$$\mathcal Ff(\omega) = (2\pi)^{-1/2}\int_{\mathbb R} e^{-i \omega t} f(t) \,dt, \quad \omega \in \mathbb R\, .$$

It is well-known that the Fourier transform preserves the functional form of a Gaussian; particularly, if $g_1(t) = e^{-t^2/2}$, then $\mathcal{F}g_1 = g_1$. When restricted to $L^1(\mathbb R) \cap L^2(\mathbb R)$ the Fourier transform can be shown to be an isometry for the $L^2$ norm with a dense range in $L^2(\mathbb R)$ \cite{stein}, and thus $\mathcal F$ extends to a unique unitary operator on $L^2(\mathbb R)$. Abusing notation, we denote the integral transform and the associated unitary operator with the same symbol $\mathcal F$. In terms of the inner product and any two functions $f$ and $g$ in $L^2(\mathbb R)$, the unitarity is expressed as $\langle \mathcal F f,  g \rangle = \langle  f, \mathcal F^{-1} g\rangle$.

The Fourier transform also enjoys a dilation property: if $\alpha\in\R\setminus\{0\}$ and $\mathcal{D}_{\alpha}f(t) = \sqrt{|\alpha|}f(\alpha t)$, then $\mathcal{F}\mathcal{D}_{\alpha} = \mathcal{D}_{\alpha^{-1}}\mathcal{F}$. Furthermore, the Fourier transform satisfies $\mathcal{F}^2f(t) = f(-t)$ and $\mathcal{F}^4 = I$ so that $\mathcal{F}^* = \mathcal{F}^{-1} = \mathcal{F}^3$. Indeed, as suggested by this periodicity relation, its eigenvalues and spectrum are comprised of $\{\pm 1, \pm i\}$.

The main results in this paper establish that there is a family of integral transforms $\{\Phi_n\}_{n=1}^\infty$, each $\Phi_n$ densely defined on $ L^2(\mathbb R)$, which 
generalize the properties of the Fourier transform in the following way:
\begin{enumerate}
\item \label{item:gn} If $g_n(t)=e^{-\frac{t^{2n}}{2n}}$, $n \in \mathbb N$, then $\Phi_n g_n = g_n$.
\item \label{item:dilation} If $\alpha\in\R\setminus\{0\}$ and $\mathcal{D}_{\alpha}$ is the dilation operator given above, then $\Phi_n\mathcal{D}_{\alpha} = \mathcal{D}_{\alpha^{-1}}\Phi_n$.
\item \label{item:Phiniso} The operator $\Phi_n$ 
is unitary and can be defined as an integral transform when its domain is suitably restricted to a dense set in $L^2(\mathbb R)$.
\item \label{item:spec} $\Phi_n^4 = I$ and its eigenvalues and spectrum are comprised only of $\pm 1, \pm i$.
\end{enumerate}

The Gaussian is a special case of a family of the Gaussian-like functions $\{g_n\}_{n=1}^\infty$ featured in property (\ref{item:gn}). The guiding principle for this paper is to retain as many properties of the Fourier transform as possible while demanding that $\Phi_n$ leaves $g_n$ invariant. Based on these axioms, we derive that $\Phi_n$ is defined as an integral transform,
\begin{equation*}
 \Phi_ng(\omega) = \int_{-\infty}^{\infty} \varphi_n(\omega t)g(t)\,dt, \quad \text{a.e. }\omega\in\R,
\end{equation*}

\noindent for each sufficiently regular function $g$ with the integral kernel $\varphi_n(\omega t) = c_n(\omega t) + i s_n(\omega t)$,
\begin{equation*}
 c_n(\eta) = \frac{1}{2}|\eta|^{n-\frac{1}{2}} J_{-1+\frac{1}{2n}}\left(\frac{|\eta|^n}{n}\right),
\end{equation*}

\noindent and
\begin{equation*}
 s_n(\eta) = -\frac{1}{2}\sgn(\eta)|\eta|^{n-\frac{1}{2}}J_{1-\frac{1}{2n}}\left(\frac{|\eta|^n}{n}\right).
\end{equation*}

\noindent The Fourier transform emerges as the special case $\mathcal{F} = \Phi_1$. The functions $c_n$ and $s_n$ are shown to be solutions to the eigenvalue equation of a (singular) Laplacian, $-\frac{d}{d\eta}\frac{1}{\eta^{2n-2}}\frac{d}{d\eta}g = g$. This shows that the asymptotic oscillatory behavior of the kernel can be tuned by choosing $n$, which is expected to be useful in applications of signal analysis to functions with chirp-like components. In addition to introducing the transforms $\Phi_n$ and their properties, we show that they give rise to associated short-time transforms which generalize the short-time Fourier transform.

The remainder of this paper is organized as follows. In Section $2$, we state a general fact about integral transforms satisfying the dilation property, develop the family of integral transforms and their kernels and show property (\ref{item:gn}). In Section $3$, we demonstrate some eigenfunctions of the transforms, discuss a function space on which $\Phi_n$ is defined and establish property (\ref{item:dilation}) for $\Phi_n$. In Section $4$, we establish property (\ref{item:Phiniso}) for the family of integral transforms, extend $\Phi_n$ to an isometry on all of $L^2(\R)$, show that the family is unitary on $L^2(\R)$ and show property (\ref{item:spec}). In Section $5$, we develop short-time analogues for the family of integral transforms and establish some identical results to those of the short-time Fourier transform.

\section{A Family of Integral Transforms}

We first want to show that a densely defined, bounded integral operator on $L^2(\R)$ satisfying the dilation property given in (\ref{item:dilation}) has an integral kernel, $\varphi$, of the form $\varphi(\omega, t) = f(\omega t)$ for some function $f$.

\begin{prop}
 Suppose $\mathcal{T}$ is a bounded integral operator defined on a dense subspace $\mathfrak{X}$ of $L^2(\R)$ such that $\mathfrak{X}$ is invariant under each $\mathcal{D}_{\alpha}$, $\alpha\neq 0$, and $\mathcal{T}\mathcal{D}_{\alpha}g = \mathcal{D}_{\alpha^{-1}}\mathcal{T}g$ for all $g\in\mathfrak{X}$. If $\varphi$ is a kernel for $\mathcal{T}$, then we can choose $\varphi$ to be of the form $\varphi(\omega, t) = f(\omega t)$ for some function $f$.
\end{prop}

\begin{proof}
 Suppose $g\in\mathfrak{X}$ and let $\alpha\in\R\setminus\{0\}$. $\mathcal{T}\mathcal{D}_{\alpha}g$ and $\mathcal{D}_{\alpha^{-1}}\mathcal{T}g$ are then defined almost everywhere. From the dilation property and by a change of variables, we have
 \begin{equation}
  \int_{-\infty}^{\infty} \varphi(\omega,\alpha t)g(t)\,dt = \int_{-\infty}^{\infty} \varphi(\omega\alpha, t)g(t)\,dt,\quad \text{a.e. } \omega\in\R.
 \end{equation}

 \noindent This identity also holds almost everywhere when choosing $g$ among a countable, dense subset of $\mathfrak{X}$, which is dense in $L^2(\R)$, and thus for a fixed $\alpha\neq 0$, $\varphi(\omega, \alpha t) = \varphi(\omega\alpha, t)$ for almost every $\omega,t\in\R$. Next, this identity is valid when selecting $\alpha$ from a countable set, say the rationals. Scaling then gives $\varphi(\omega\alpha^{-1},\alpha t) = \varphi(\omega, t)$ almost everywhere, which shows that the left hand side does not depend on $\alpha$. Now taking the limit $\alpha\to\omega$ through the rationals gives that $\varphi(\omega, t)$ is a function of $\omega t$, defined almost everywhere.
\end{proof}

Next we define the family of $n$-Gaussians.

\begin{defn}
For $n \in \mathbb N$, the $n$-Gaussian is the function $g_n \in L^2(\mathbb R)$ such that
$$
   g_n(t) = e^{-\frac{t^{2n}}{2n}} \, .
$$ 
\end{defn}

In analogy with the Fourier transform, we require in (\ref{item:gn}) that $g_n$ be invariant under the integral operator $\Phi_n$. The $n$-Gaussians behave as nearly ideal low-pass filters and as such are natural candidates for defining an integral transform. We denote the kernel of $\Phi_n$ by $\varphi_n$.

Our goal is to devise a family of operators that generalize the Fourier transform while retaining as many of its properties as possible. Specifically we look to an axiomatic characterization of the transforms. A similar axiomatic approach has been developed for the case of the Hilbert transform \cite{boche}. To this end, we inspect some properties of the Fourier kernel.

The most obvious properties of the Fourier kernel are that the real part is even, the imaginary part is odd and it is real analytic in both variables. There are multiple ways in which sine and cosine are related: as derivatives of each other, as distributional Hilbert transforms of each other and as linearly independent eigenfunctions of the Laplacian in one variable. The derivative operator does not play well with unitarity so this is not feasible; the distributional Hilbert transform, while rich with theory, is difficult to treat in practice. For these reasons, and more, we choose to use view sine and cosine as linearly independent solutions to the same differential equation.

We write $\varphi_n$ as
\begin{equation}
 \varphi_n(\omega,t) = c_n(\omega,t)+is_n(\omega,t),
\end{equation}

\noindent where $c_n$ and $s_n$ are real-valued. With these properties in mind, we state the following assumptions for $\varphi_n$:

\begin{enumerate}[(a)]
 {\indentitem \item $\varphi_n$ is of the form $\varphi_n(\omega,t) = f(\omega t)$ for some complex-valued function $f$,}
 {\indentitem \item $\varphi_n$ is real analytic,}
 {\indentitem \item $c_n$ is even and $s_n$ is odd,} \label{eq:phieo}
 {\indentitem \item $c_n$ and $s_n$ are linearly independent solutions to the same differential equation.}
\end{enumerate}

\noindent With the stipulated form for the integral kernel, \eqref{item:gn} becomes
\begin{equation}\label{eq:nGaussinv}
e^{-\frac{\omega^{2n}}{2n}} = \int_{-\infty}^{\infty}\varphi_n(\omega t)e^{-\frac{t^{2n}}{2n}}\,dt.
\end{equation}

From this integral expression we can deduce some immediate consequences for $\varphi_n$. Without assuming evenness of $c_n$, it could not be uniquely determined from \eqref{eq:nGaussinv} as any odd, slowly-growing function can be added to it and the integration against the $n$-Gaussian would be unchanged. Additionally, $s_n$ must be orthogonal to the $n$-Gaussians for all $\omega\in\R$, otherwise the right side of \eqref{eq:nGaussinv} would be complex whereas the left side is pure real. These observations support assumption (c) for $\varphi_n$. For now we will consider $c_n$ as it can be easily established from \eqref{eq:nGaussinv}.

\begin{defn}
 Let $n\in\mathbb{N}$, $l\in\mathbb{N}_0$ and
\begin{equation} \label{eq:cn_series}
 c(n;l) = \frac{(-1)^ln}{(2n)^{2l+\frac{1}{2n}}\Gamma\left(l+\frac{1}{2n}\right)l!}. \nonumber
\end{equation}

\noindent We then define $c_n$ as the entire function with the series
\begin{equation}\label{eq:cn}
 c_n(\eta) = \sum_{l=0}^{\infty} c(n;l)\eta^{2nl} \, , \qquad \eta \in \mathbb C.
\end{equation}
\end{defn}

\begin{lem} 
 Let the function $c_n$ be as in \eqref{eq:cn}, then, for $\omega\in\R$, it is real analytic and satisfies the integral equation
 \begin{equation*}
 e^{-\frac{\omega^{2n}}{2n}} = \int_{-\infty}^{\infty} c_n(\omega t)e^{-\frac{t^{2n}}{2n}} \,dt
\end{equation*}
\end{lem}

\begin{proof}
Substituting the stipulated form for $c_n$ into the integral equation gives
\begin{equation}
 e^{-\frac{\omega^{2n}}{2n}} = \int_{-\infty}^{\infty} \sum_{l=0}^{\infty} c(n;l)\omega^{2nl}t^{2nl} e^{-\frac{t^{2n}}{2n}} \,dt. \label{eq:sum}
\end{equation}
If the series given by
\begin{equation}\label{eq:fubini}
 \sum_{l=0}^{\infty} c(n;l) \omega^{2nl}\int_{-\infty}^{\infty}t^{2nl}e^{-\frac{t^{2n}}{2n}} \,dt
\end{equation}
converges absolutely for all $\omega\in\R$, then the integral and summation in \eqref{eq:sum} can be interchanged by the Fubini-Tonelli theorem. Substituting $y=\frac{t^{2n}}{2n}$ in the integral in \eqref{eq:fubini} yields
\begin{eqnarray}
 \int_{-\infty}^{\infty} t^{2nl}e^{-\frac{t^{2n}}{2n}} \,dt &=& \frac{1}{n}(2n)^{l+\frac{1}{2n}}\int_0^{\infty} y^{l+\frac{1}{2n}-1}e^{-y}\,dy\nonumber\\
&=& \frac{1}{n}(2n)^{l+\frac{1}{2n}}\Gamma\left(l+\frac{1}{2n}\right).\nonumber
\end{eqnarray}
Inserting this expression into \eqref{eq:fubini} yields the following series
\begin{equation}
 \sum_{l=0}^{\infty}\frac{1}{l!}\left(-\frac{\omega^{2n}}{2n}\right)^l.
\end{equation}

This series converges absolutely for all $\omega\in\R$ and so the integration and summation can be interchanged in \eqref{eq:sum}, resulting in $e^{-\frac{\omega^{2n}}{2n}}$ and thus the lemma is proved.
\end{proof}

In fact, for real $\eta$, $c_n$ has the closed-form expression:
\begin{equation}
 c_n(\eta) = \frac{1}{2}|\eta|^{n-\frac{1}{2}}J_{-1+\frac{1}{2n}}\left(\frac{|\eta|^n}{n}\right),
\end{equation}
where $J_{\nu}$ is the Bessel function of the first kind of order $\nu$ \cite[p. 40]{watson}. This can be checked directly by manipulating the Bessel function power series:
\begin{equation*}
 J_{\nu}(z) = \sum_{m=0}^{\infty} \frac{(-1)^m}{\Gamma(m+\nu+1)m!}\left(\frac{z}{2}\right)^{2m+\nu}.
\end{equation*}

Since $c_n$ can be expressed in terms of a Bessel function of the first kind, one may expect that $c_n$ is the solution to a second-order differential equation. We prove this in the next proposition.

\begin{prop}
 The function $c_n$ as defined in \eqref{eq:cn} is a solution to the differential equation
 \begin{equation}
  -\frac{d}{d\eta}\frac{1}{\eta^{2n-2}}\frac{d}{d\eta}c_n(\eta) = c_n(\eta).
 \end{equation}
\end{prop}

\begin{proof}
 Since the series defined in \eqref{eq:cn} converges uniformly on compact sets, we can differentiate the series term-by-term. Hence we have that
 \begin{eqnarray*}
  -\frac{d}{d\eta}\frac{1}{\eta^{2n-2}}\frac{d}{d\eta}c_n(\eta) &=& \sum_{l=0}^{\infty} \frac{(-1)^ln}{(2n)^{2l+\frac{1}{2n}}\Gamma\left(l+\frac{1}{2n}\right)l!}\left(-\frac{d}{d\eta}\frac{1}{\eta^{2n-2}}\frac{d}{d\eta}\right)\eta^{2nl} \\
  &=& \sum_{l=1}^{\infty} \frac{(-1)^l n}{(2n)^{2l+\frac{1}{2n}}\Gamma\left(l+\frac{1}{2n}\right)l!}(-2nl)(2nl-2n+1)\eta^{2nl-2n}.
 \end{eqnarray*}
 
Upon reindexing the series and making use of the recursive property of the gamma function, this becomes $c_n(\eta)$ as claimed.
\end{proof}

In fact, a more general property holds. If $\D_n$ denotes the operator $-\frac{d}{dt}\frac{1}{t^{2n-2}}\frac{d}{dt}$, defined on sufficiently regular entire functions, then for fixed $\omega\in\R$, $\D_n(c_n(\omega t)) = \omega^{2n}c_n(\omega t)$. This equality can be checked in a similar manner as above.

Using the above result and the assumptions that $s_n$ is real analytic and satisfies the same differential equation as $c_n$, we can now derive $s_n$---at least up to a multiplicative factor. We do so in the next proposition.

\begin{prop}
 The function $f$ defined by the series
 \begin{equation*}
  f(\eta) = \sum_{l=0}^{\infty} \frac{(-1)^l}{(2n)^{2l}\Gamma\left(l+2-\frac{1}{2n}\right)l!}\eta^{2nl+2n-1}
 \end{equation*}
 solves the differential equation $\D_nf(\eta) = f(\eta)$. Moreover, the solution set of entire functions to $\D_n g = g$ is spanned by $c_n$ and $f$ as defined above.
\end{prop}

\begin{proof}
 That $\D_n f = f$ follows via the same arguments in Proposition $2$. Moreover, suppose that $g$ is a solution to the differential equation $\D_n g = g$ and is given by the power series
 \begin{equation*}
  g(\eta) = \sum_{l=0}^{\infty} \alpha_l \eta^l
 \end{equation*}
for some $\alpha_l\in\R$. Then comparing the terms in the power series of $\D_ng$ and $g$, we get from
 \begin{eqnarray*}
  -\frac{d}{d\eta}\frac{1}{\eta^{2n-2}}\frac{d}{d\eta}g(\eta) &=& -\frac{d}{d\eta}\frac{1}{\eta^{2n-2}}\frac{d}{d\eta}\sum_{l=0}^{\infty} \alpha_l \eta^l \\
  &=& -\sum_{l=1}^{2n-2} \alpha_l l(l-2n+1)t^{l-2n} - \sum_{l=2n}^{\infty} \alpha_l l(l-2n+1)t^{l-2n}
 \end{eqnarray*}
 
 \noindent that $\alpha_1=\cdots = \alpha_{2n-2} = \alpha_{2n+1} = \cdots = \alpha_{4n-2} = \cdots = 0$. Thus the only nonzero coefficients are $\alpha_{2nl}$ and $\alpha_{2nl-1}$ for some $l$. By solving the recursion relations for the coefficients, it follows that the solutions to $\D_n g = g$ are linear combinations of $c_n$ and $f$ since $c_n$ and $f$ are linearly independent solutions to the same second order differential equation.
\end{proof}

With this result established, we make the following definition for $s_n$.

\begin{defn} 
 Let $n\in\mathbb{N}$, $l\in\mathbb{N}_0$ and
\begin{equation} \label{eq:sn_series}
 s(n;l) = -\frac{(-1)^ln}{(2n)^{2l+2-\frac{1}{2n}}\Gamma\left(l+2-\frac{1}{2n}\right)l!}. \nonumber
\end{equation}

\noindent We then define $s_n$ as the entire function with the series
\begin{equation}\label{eq:sn}
 s_n(\eta) = \sum_{l=0}^{\infty} s(n;l)\eta^{2nl+2n-1} \, , \qquad \eta \in \mathbb C.
\end{equation}
\end{defn}

As noted above, $s_n$ is only unique up to a multiplicative factor. The choice of $s_n$ above guarantees unitarity; in fact, the only other choice that gives unitarity is $-s_n$, in exact agreement with Fourier transform theory.

Like $c_n$, $s_n$ has a convenient representation in terms of a Bessel function of the first kind; particularly, we have that $s_n(\eta) = -\frac{1}{2}\sgn(\eta)|\eta|^{n-\frac{1}{2}}J_{1-\frac{1}{2n}}\left(\frac{|\eta|^n}{n}\right)$. If $n=1$, $\varphi(\eta) = \frac{1}{\sqrt{2\pi}}e^{-i\eta}$ as expected.

With these representations in terms of Bessel functions, we can inspect the asymptotic behavior of $\varphi_n$ easily. The Bessel function $J_{\nu}$ has the following asymptotic form \cite[p. 199]{watson}:
\begin{equation*}
 J_{\nu}(\eta) \sim \sqrt{\frac{2}{\pi\eta}}\cos\left(\eta - \frac{\nu\pi}{2} - \frac{\pi}{4}\right) + O\left(z^{-\frac{3}{2}}\right).
\end{equation*}

Hence $c_n$ and $s_n$ have the following asymptotic forms which will be useful in the analysis in the next section:
\begin{equation} \label{eq:cn_asym}
 c_n(\eta) \sim \sqrt{\frac{n}{2\pi}}|\eta|^{\frac{n-1}{2}}\cos\left(\frac{|\eta|^n}{n}+\frac{\pi}{4}\left(1-\frac{1}{n}\right)\right) + O\left(|\eta|^{-\frac{n+1}{2}}\right),
\end{equation}
\begin{equation} \label{eq:sn_asym}
 s_n(\eta) \sim \sqrt{\frac{n}{2\pi}}\sgn(\eta)|\eta|^{\frac{n-1}{2}}\cos\left(\frac{|\eta|^n}{n}-\frac{\pi}{4}\left(3-\frac{1}{n}\right)\right) + \sgn(\eta)O\left(|\eta|^{-\frac{n+1}{2}}\right).
\end{equation}

\section{The $\Phi_n$ transform and its domain}
\subsection{Developing the $\Phi_n$ transform}

When developing the Fourier transform in full generality, it is often first defined on functions in $L^1(\R)$ and then extended by considering limits of Cauchy sequences in the dense subset $L^1(\R)\cap L^2(\R)$ or $\mathcal{S}(\R)$ of $L^2(\R)$. For such functions, the results from the theory on $L^1(\R)$ are true as well which streamlines many proofs. We employ a similar approach in the present setting with a caveat: because the kernels diverge at infinity, the function space on which the integral transforms are defined cannot be all of $L^1(\R)$ but must be modified to mollify the growth of $\varphi_n$ at infinity.


Let $d\mu_n(t) = |t|^{\frac{n-1}{2}}\,dt$. We claim that for $f\in L^1(\R,dt)\cap L^1(\R,d\mu_n)$, $\int_{\R}|\varphi_n(\omega t)f(t)|\,dt$ is finite. In the case of $n=1$, this space is identically $L^1(\R)$ which is the usual space upon which the Fourier transform is defined. Let $\omega\in\R$ be fixed, $f\in L^1(\R,dt)\cap L^1(\R,d\mu_n)$ and $R\gg 0$, then
\begin{align*}
 \int_{-\infty}^{\infty} |\varphi_n(\omega t)f(t)| \,dt &= \int_{|t|\le R}|\varphi_n(\omega t)| |f(t)|\,dt + \int_{|t|> R}|\varphi_n(\omega t)||f(t)|\,dt \\
 &\le M_1\int_{|t|\le R}|f(t)|\,dt + \sqrt{\frac{n}{2\pi}} |\omega|^{\frac{n-1}{2}}\int_{|t|>R} \left(|t|^{\frac{n-1}{2}}+O\left(|t|^{-\frac{n+1}{2}}\right) \right)|f(t)|\,dt.
\end{align*}

In the first term, we have used the fact that $\varphi_n$ is continuous and hence bounded on compact sets. The first integral is then finite since $f\in L^1(\R,dt)$. In the second term, we have used the asymptotic form for $\varphi_n$ as per \eqref{eq:cn_asym} and \eqref{eq:sn_asym}. The integral of $|f|$ against $|t|^{\frac{n-1}{2}}$ in the second term is finite since $f\in L^1(\R,d\mu_n)$ by hypothesis. Moreover the integral of $|f|$ against $O(|t|^{-\frac{n+1}{2}})$ in the second term is finite since for some $M_2 > 0$, $O(|t|^{-\frac{n+1}{2}})|f(t)| \le M_2 R^{-\frac{n+1}{2}}|f(t)|$ and $f\in L^1(\R,dt)$. Thus for $f\in L^1(\R,dt)\cap L^1(\R,d\mu_n)$, $\omega\mapsto\int_{\R}\varphi_n(\omega t)f(t)\,dt$ is defined pointwise.

Since we are ultimately interested in an $L^2$ theory, it stands to reason that we should consider the space $L^1(\R,dt)\cap L^1(\R,d\mu_n)\cap L^2(\R,dt)$. It is well-known that if $f\in L^1(\R,dt)\cap L^2(\R,dt)$, then $\mathcal{F}f\in L^2(\R,dt)$; however this is not obviously true in general. Thus the natural function space upon which $\Phi_n$ acts, denoted $\dom \Phi_n$, is given by
\begin{equation} \label{eq:phi_dom}
 \dom \Phi_n = \{f\in L^1(\R,dt)\cap L^1(\R,d\mu_n)\cap L^2(\R,dt):\omega\mapsto \int_{\R} \varphi_n(\omega t) f(t)\,dt\in L^2(\R,dt)\}.
\end{equation}

\noindent This is clearly a vector space however we postpone discussion of its density in $L^2(\R)$. With a formal domain, we may now define the $\Phi_n$ transform.

\begin{defn}
 Let $f\in \dom\Phi_n$ and $\omega\in\R$, then $\Phi_n f$ is defined pointwise by
 \begin{equation}
  \Phi_nf(\omega) = \int_{-\infty}^{\infty} \varphi_n(\omega t)f(t)\,dt.
 \end{equation}
\end{defn}

\noindent Clearly the dilation property (\ref{item:dilation}) holds for $f\in\dom\Phi_n$ which a simple change of variable shows. Before showing analyic properties of $\Phi_n$, we first explore some of its eigenfunctions as these will play an important role in the $L^2$ theory for $\Phi_n$.

\subsection{Some eigenfunctions of $\Phi_n$}

We have already demonstrated one eigenfunction for $\Phi_n$: $g_n$. From this, we can extract a family of eigenfunctions for $\Phi_n$ by implementing Akhiezer's technique since the kernel of $\Phi_n$ is of the form $\varphi_n(\omega, t) = f(\omega t)$. Since $g_n$ is an eigenfunction of $\Phi_n$ by hypothesis,
\begin{equation*}
 e^{-\frac{\omega^{2n}}{2n}} = \int_{-\infty}^{\infty}\varphi_n(\omega t)e^{-\frac{t^{2n}}{2n}}\,dt.
\end{equation*}
Making the changes of variables $t = \alpha^{\frac{1}{2n}}x$ and $\omega = \alpha^{-\frac{1}{2n}}y$ where $\alpha > 0$, we see that $\varphi_n$ is unchanged but we have
\begin{equation*}
 e^{-\frac{y^{2n}}{2n\alpha}} = \int_{-\infty}^{\infty} \varphi_n(xy)e^{-\alpha\frac{x^{2n}}{2n}}\alpha^{\frac{1}{2n}}\,dx.
\end{equation*}
\noindent Multiplying both sides by $\alpha^{-\frac{1}{4n}}$ yields the following
\begin{equation*}
 \alpha^{-\frac{1}{4n}}e^{-\frac{y^{2n}}{2n\alpha}} = \int_{-\infty}^{\infty}\varphi_n(xy)e^{-\alpha\frac{x^{2n}}{2n}}\alpha^{\frac{1}{4n}}\,dx.
\end{equation*}
\noindent We introduce the parameter $\beta = \frac{1}{\alpha}$ and note that $\alpha\frac{\partial}{\partial\alpha} = -\beta\frac{\partial}{\partial\beta}$. Thus
\begin{equation}
 \left(-\beta\frac{\partial}{\partial\beta}\right)^m\left(\beta^{\frac{1}{4n}}e^{-\beta\frac{y^{2n}}{2n}}\right) = \int_{-\infty}^{\infty}\varphi_n(xy)\left(\alpha\frac{\partial}{\partial\alpha}\right)^m\left(\alpha^{\frac{1}{4n}}e^{-\alpha\frac{x^{2n}}{2n}}\right)\,dx.
\end{equation}

To eliminate the dependence upon the parameters $\alpha$ and $\beta$, after differentiating they may be set to $1$. It is then clear that the even eigenfunctions are
\begin{equation}
 \phi_{2m}^{(n)}(t) = \left(\alpha\frac{\partial}{\partial\alpha}\right)^m\left(\alpha^{\frac{1}{4n}}e^{-\alpha\frac{t^{2n}}{2n}}\right)\bigg|_{\alpha=1}, \label{eq:evenefuncs}
\end{equation}
\noindent with eigenvalue $(-1)^m$. Particularly, $\Phi_n^2$ acts as the identity on these functions.

Taking cues from the Fourier transform, the Hermite-Gauss functions, and noting that the lowest power in the series for $s_n(\eta)$ is $\eta^{2n-1}$, the obvious candidate for an odd eigenfunction of $\Phi_n$ is $t^{2n-1}e^{-\frac{t^{2n}}{2n}}$. To see that this is indeed an eigenfunction of $\Phi_n$, note that
\begin{equation*}
 \int_{-\infty}^{\infty} \varphi_n(\omega t)t^{2n-1}e^{-\frac{t^{2n}}{2n}}\,dt = -i\sgn(\omega)|\omega|^{n-\frac{1}{2}} \int_0^{\infty} t^{3n-\frac{3}{2}}J_{1-\frac{1}{2n}}\left(\frac{|\omega|^n}{n}t^n\right)e^{-\frac{t^{2n}}{2n}}\,dt.
\end{equation*}

\noindent Letting $z = t^n$, this becomes
\begin{equation*}
 \int_{-\infty}^{\infty} \varphi_n(\omega t)t^{2n-1}e^{-\frac{t^{2n}}{2n}}\,dt = -\frac{i}{n}\sgn(\omega)|\omega|^{n-\frac{1}{2}} \int_0^{\infty} z^{2-\frac{1}{2n}}J_{1-\frac{1}{2n}}\left(\frac{|\omega|^n}{n}z\right)e^{-\frac{z^2}{2n}}\,dz.
\end{equation*}

\noindent This integral simplifies nicely \cite[p. 394]{watson} to give
\begin{equation*}
 \int_{-\infty}^{\infty} \varphi_n(\omega t)t^{2n-1}e^{-\frac{t^{2n}}{2n}}\,dt = -i \omega^{2n-1}e^{-\frac{\omega^{2n}}{2n}}.
\end{equation*}

Hence $t^{2n-1}e^{-\frac{t^{2n}}{2n}}$ is an eigenfunction of $\Phi_n$ with eigenvalue $-i$. Repeating the same analysis as above with the even eigenfunctions, we obtain the following odd eigenfunctions with eigenvalue $(-1)^{m+1}i$:
\begin{equation}
 \phi_{2m+1}^{(n)}(t) = t^{2n-1}\left(\alpha\frac{\partial}{\partial\alpha}\right)^m\left(\alpha^{1-\frac{1}{4n}}e^{-\alpha\frac{t^{2n}}{2n}}\right)\bigg|_{\alpha=1}.\label{eq:oddefuncs}
\end{equation}

\noindent Unlike in the case of the even eigenfunctions, $\Phi_n^2$ acts as the negative identity on the odd eigenfunctions.

Note that $\phi_m^{(n)}\in\dom\Phi_n$ for all $m$ and $n$. Moreover $\phi_m^{(n)}$ has eigenvalue $(-i)^m$ under $\Phi_n$. Since $\varphi_n$ has polynomial growth and is continuous, $|\varphi_n(\eta)| \le M_1 + M_2|\eta|^{\alpha}$ for some $M_1,M_2,\alpha > 0$. Noting that $\phi_m^{(n)}$ has exponential decay, it follows that
\begin{equation*}
 \int_{-\infty}^{\infty}\int_{-\infty}^{\infty}|\varphi_n(\omega t)\phi_m^{(n)}(t)\phi_{m'}^{(n)}(\omega)|\,dt\,d\omega \le \int_{-\infty}^{\infty} \int_{-\infty}^{\infty} (M_1 + M_2|\omega t|^{\alpha})|\phi_m^{(n)}(t)\phi_{m'}^{(n)}(\omega)|\,dt\,d\omega < \infty.
\end{equation*}

\noindent Hence by Fubini-Tonelli, we have that
\begin{equation*}
 \langle \phi_m^{(n)},\phi_{m'}^{(n)}\rangle = (-i)^m\langle\Phi_n\phi_m^{(n)},\phi_{m'}^{(n)}\rangle = (-i)^m\langle\phi_m^{(n)}, \overline{\Phi_n\phi_{m'}^{(n)}}\rangle = (-i)^{m-m'}\langle \phi_m^{(n)},\phi_{m'}^{(n)}\rangle,
\end{equation*}

\noindent and so if $m\not\equiv m' \pmod 4$, then $\langle \phi_m^{(n)},\phi_{m'}^{(n)}\rangle = 0$. This is in direct analogy with the traditional Fourier transform eigenfunctions: there are four mutually orthogonal eigenspaces.

Furthermore, $\{\phi_m^{(n)}\}$ is a complete set of eigenfunctions. To see this, note that $\phi_m^{(n)}$ is of the form $p_m^{(n)}(t)e^{-\frac{t^{2n}}{2n}}$, where $p_m^{(n)}$ is a polynomial of degree $2nk$ or $2nk-1$; moreover, $p_m^{(n)}$ is a linear combination of powers of the form $t^{2nl}$ or $t^{2nl-1}$, depending on whether $m$ is even or odd.

We can employ Gram-Schmidt to obtain an orthonormal set from the eigenfunctions; the orthonormal set is denoted by $\widetilde{p}_m^{(n)}(t)e^{-\frac{t^{2n}}{2n}}$, where $\widetilde{p}_m^{(n)}$ is a polynomial of degree $2nk$ or $2nk-1$---in general, $p_m^{(n)}$ and $\widetilde{p}_m^{(n)}$ need not be the same. Additonally, the Gram-Schmidt procedure only occurs within each eigenspace since the different eigenspaces are mutually orthogonal by the preceding argument.

Since $p_{2k}^{(n)}$ is comprised of powers $t^{2nl}$, we can view $\widetilde{p}_{2k}^{(n)}(t)$ as a polynomial $\widetilde{q}_{2m}^{(n)}(t^{2n})$. The orthogonality of the functions $\widetilde{p}_{2k}^{(n)}(t)e^{-\frac{t^{2n}}{2n}}$ can then be summarized as
\begin{equation*}
 \int_{-\infty}^{\infty} \widetilde{p}_{2k}^{(n)}(t)\widetilde{p}_{2l}^{(n)}(t)e^{-\frac{t^{2n}}{n}}\,dt = 2n^{1-\frac{1}{2n}}\delta_{kl}.
\end{equation*}

\noindent After a change of variable, this becomes
\begin{equation*}
 \int_0^{\infty} t^{-1+\frac{1}{2n}}\widetilde{q}_{2k}^{(n)}(nt)\widetilde{q}_{2l}^{(n)}(nt) e^{-t}\,dt = \delta_{kl}.
\end{equation*}

Proceeding in the same way for the odd eigenfunctions, we can view $\widetilde{p}_{2k+1}^{(n)}(t)$ as a polynomial $t^{2n-1}\widetilde{q}_{2k+1}^{(n)}(t^{2n})$. The orthogonality relation can again be summarized as
\begin{equation*}
 \int_{-\infty}^{\infty} \widetilde{p}_{2k+1}^{(n)}(t)\widetilde{p}_{2l+1}^{(n)}(t)e^{-\frac{t^{2n}}{n}}\, dt = \frac{1}{2}n^{-3+\frac{1}{n}}\delta_{kl}.
\end{equation*}

\noindent After making a change of variable, this becomes
\begin{equation*}
 \int_0^{\infty} t^{3-\frac{1}{n}} \widetilde{q}_{2k+1}^{(n)}(nt)\widetilde{q}_{2l+1}^{(n)}(nt) e^{-t}\,dt = \delta_{kl}.
\end{equation*}

Since $\widetilde{q}_k$ is a polynomial, the analysis by Akhiezer \cite[p. 61]{akhiezer} for the completeness of the Laguerre polynomials proves the completeness of eigenfunctions $\{\phi_m^{(n)}\}$ in $L^2(\R)$. The completeness of the Laguerre polynomials can be summarized as follows:
\begin{thm}
 Let $f:[0,\infty)\to\R$ be measurable and $\nu>-1$, then if
 \begin{enumerate}
  \item $\displaystyle\int_0^{\infty} e^{-x} x^{\nu} |f(x)|^2\,dx <\infty$,
  \item $\displaystyle\int_0^{\infty} e^{-x} x^{\nu} f(x) x^m\,dx = 0$
 \end{enumerate}
 
 \noindent for all $m\in\mathbb{N}_0$, then $f\equiv 0$.
\end{thm}

Furthermore, there is a convenient recursion relation for the $\phi_m^{(n)}$ which follows from \eqref{eq:evenefuncs} and \eqref{eq:oddefuncs}:
\begin{equation}
 \phi_{m+2}^{(n)}(t) = \frac{1}{4n}\phi_m^{(n)}(t) + \frac{t}{2n}\frac{d\phi_m^{(n)}}{dt}.
\end{equation}

Since the eigenfunctions $\phi_m^{(n)}$ of $\Phi_n$ are complete and $\phi_m^{(n)}\in L^1(\R,dt)\cap L^1(\R,d\mu_n)\cap L^2(\R,dt)$, it follows that $\dom \Phi_n$ is dense in $L^2(\R)$.

\section{Properties of $\Phi_n$}
\subsection{Preservation of the $L^2$ norm}

We wish to show that $\Phi_n$ is an $L^2$ isometry on $\dom\Phi_n$. Traditionally, the $L^2$ isometry of the Fourier transform from $L^1(\R,dt)\cap L^2(\R,dt)$ to $L^2(\R,dt)$ is proved by appealing to the convolution theorem. However no obvious convolution theorem exists for $\Phi_n$ in general and so we take a purely $L^2$ approach by appealing to the completeness of the eigenfunctions of $\Phi_n$.

\begin{thm}
 If $f\in \dom\Phi_n$, $\|\Phi_n f\|_{L^2(\R,dt)} = \|f\|_{L^2(\R,dt)}$, so $\Phi_n$ is an isometry with dense range and extends to a unitary on $L^2(\R,dt)$.
\end{thm}

\begin{proof}
 Let $\{\psi_m^{(n)}\}$ be an orthonormal basis of eigenfunctions of $\Phi_n$. Such a basis exists by the analysis in Section $3.2$. For $f\in\dom\Phi_n$, $\Phi_n f\in L^2(\R,dt)$ by hypothesis and so $\langle \Phi_nf,\psi_m^{(n)}\rangle$ is finite. Thus
 \begin{align*}
  \langle \Phi_n f,\psi_m^{(n)}\rangle &= \int_{-\infty}^{\infty} \Phi_nf(\omega)\overline{\psi_m^{(n)}(\omega)}\,d\omega \\
  &= \int_{-\infty}^{\infty}\left(\int_{-\infty}^{\infty} \varphi_n(\omega t)f(t)\,dt\right)\overline{\psi_m^{(n)}(\omega)}\,d\omega.
 \end{align*}

 We can interchange the integrals above since $\omega\mapsto\int_{\R} |\varphi_n(\omega t)f(t)|\,dt$ is finite everywhere and has at most polynomial growth at infinity and $\psi_m^{(n)}$ has exponential decay. Therefore
 \begin{equation*}
  \langle \Phi_nf,\psi_m^{(n)}\rangle = \int_{-\infty}^{\infty} f(t)\overline{\int_{-\infty}^{\infty}\overline{\varphi_n(\omega t)}\psi_m^{(n)}(\omega) \,d\omega}\,dt.
 \end{equation*}
 
 It is clear that $\int_{\R} \overline{\varphi_n(\omega t)}\psi_m^{(n)}(\omega)\,d\omega = i^m\psi_m^{(n)}(t)$, giving $\langle \Phi_nf,\psi_m^{(n)}\rangle = (-i)^m\langle f,\psi_m^{(n)}\rangle$. If we write $f = \sum_m \langle f,\psi_m^{(n)}\rangle \psi_m^{(n)}$, then $\Phi_n f = \sum_m \langle \Phi_n f,\psi_m^{(n)}\rangle \psi_m^{(n)} = \sum_m (-i)^m\langle f,\psi_m^{(n)}\rangle \psi_m^{(n)}$. Computing the norm of $\Phi_n f$, we have $\|\Phi_n f\|_{L^2(\R, dt)}^2 = \sum_m |(-i)^m\langle f,\psi_m^{(n)}\rangle|^2 = \sum_m |\langle f,\psi_m^{(n)}\rangle|^2 = \|f\|_{L^2(\R,dt)}^2$.
 
 Hence $\Phi_n$ is an $L^2$ isometry on $\dom\Phi_n$ which is dense in $L^2(\R,dt)$ and so $\Phi_n$ extends to an isometry on $L^2(\R,dt)$. Moreover, $\Phi_n$ has dense range in $L^2(\R,dt)$ since its range includes the span of the eigenfunctions $\{\phi_m^{(n)}\}$, thus $\Phi_n$ extends to a unitary on $L^2(\R,dt)$.
 \end{proof}

  In an abuse of notation, we denote the unitary extension of $\Phi_n$ to $L^2(\R,dt)$ by $\Phi_n$ though there is no risk of confusion as the meaning will be clear from context. Since the dilation property holds on $\dom\Phi_n$, $\Phi_n$ is bounded and $\mathcal{D}_{\alpha}$ is bounded, the dilation property holds for the unitary extension of $\Phi_n$ via simple continuity arguments.
 
\subsection{The Spectrum of $\Phi_n$}

By analogy with the Fourier transform, we wish to show that $\Phi_n$ satisfies $\Phi_n^4 f= f$ for each $f \in L^2(\R)$ which in turn gives that the spectrum of $\Phi_n$ is contained in $\{\pm 1,\pm i\}$.

\begin{thm}
 $\Phi_n^4=I$ on $L^2(\R)$ and its spectrum is comprised only of $\pm 1,\pm i$.
\end{thm}

\begin{proof}
 Let $f\in L^2(\R)$ and $\{\psi_m^{(n)}\}$ be an orthonormal basis of eigenfunctions for $\Phi_n$, then
 \begin{equation*}
  \langle \Phi_n^4f, \psi_m^{(n)}\rangle = \langle f,(\Phi_n^*)^4\psi_m^{(n)}\rangle = \langle f, i^{4m}\psi_m^{(n)}\rangle = \langle f,\psi_m^{(n)}\rangle.
 \end{equation*}

 Since this holds for all $m$, it must be the case that $\Phi_n^4 f = f$, i.e. $\Phi_n^4 = I$. This gives that $\Phi_n^* = \Phi_n^{-1} = \Phi_n^3$ naturally. This generalizes the well-known result for the Fourier transform which states that $\mathcal{F}^* = \mathcal{F}^{-1} = \mathcal{F}^3$.

The spectral mapping theorem \cite{rudin} shows that the spectrum of $\Phi_n$ is contained in $\{\pm 1,\pm i\}$. In fact, in Section $3$ we demonstrated that each of these spectral values is realized and each is indeed eigenvalue.
\end{proof}

\subsection{The $\Phi_n$ and Fourier-Bessel transforms}

 With the appearance of Bessel functions in the expression for $\varphi_n$, it is natural to ask what, if any, connection there is between $\Phi_n$ and the Fourier-Bessel transform. We choose to consider the following definition for the Fourier-Bessel transform:
 \begin{equation} \label{eq:Fourier_Bessel}
  \mathcal{F}_{\nu}f(\omega) = \int_0^{\infty} j_{\nu}(\omega t)f(t)\,d\lambda_{\nu}(t),
 \end{equation}

 \noindent where $d\lambda_{\nu}(t) = t^{2\nu+1}\,dt$ and $j_{\nu}(t) = t^{-\nu}J_{\nu}(t)$. Most analysis of the Fourier-Bessel transform is restricted to the case $\nu>-\frac{1}{2}$ as in this range the measure $d\lambda_{\nu}$ is non-singular (c.f. \cite{jaming}). Some analysis has been done in the regime $-1 < \nu < -\frac{1}{2}$, cf. \cite[p. 62]{akhiezer}. $\mathcal{F}_{\nu}$ is an isometry on $L^2(\R^+,d\lambda_{\nu})$ when restricted to a dense subspace and also extends to a unitary on $L^2(\R, d\lambda_{\nu})$.
 
 Write $\Phi_n = \Phi_n^+ + i\Phi_n^-$, where $\Phi_n^+$ is the integral operator with integral kernel $c_n$ and $\Phi_n^-$ is the integral operator with integral kernel $s_n$. $\Phi_n^+$ and $\Phi_n^-$ can be thought of as restrictions of $\Phi_n$ to even and odd functions, respectively. Thus $\Phi_n$ can be written as $\Phi_n = \Phi_n^+\oplus i\Phi_n^-$, where we have decomposed $\dom\Phi_n$ into its even and odd subspaces.
 
 To relate $\Phi_n$ to $\mathcal{F}_{\nu}$ we must project functions onto $\R^+$ since the Fourier-Bessel transform is restricted to $\R^+$. Let $\mathcal{P}^+$ denote the projection onto $\R^+$. If $f\in\dom\Phi_n$ is even, then there is a natural relationship between $\mathcal{P}^+f$ and $\Phi_n f$: $\Phi_nf = \Phi_n^+f = 2\Phi_n^+ \mathcal{P}^+f$. A similar relationship holds for odd functions. Thus we may restrict our attention to those $f\in\dom\Phi_n$ with support on $\R^+$ when considering $\Phi_n$ without loss of generality.
 
 Define the operators $\mathcal{S}_n^+:L^2(\R^+,dt)\to L^2(\R,d\lambda_{-1+\frac{1}{2n}})$ and $\mathcal{S}_n^-:L^2(\R,dt)\to L^2(\R^+,d\lambda_{1-\frac{1}{2n}})$ by $\mathcal{S}_n^+ f(t) = n^{-\frac{1}{2}+\frac{1}{2n}}f(\sqrt[n]{nt})$ and $\mathcal{S}_n^-f(t) = n^{-\frac{1}{2}+\frac{1}{2n}} t^{-2+\frac{1}{n}}f(\sqrt[n]{nt})$. $\mathcal{S}_n^+$ and $\mathcal{S}_n^-$ are both invertible and their inverses are given by a simple change of variable. Furthermore, $\Phi_n^+ = (\mathcal{S}_n^+)^{-1}\mathcal{F}_{-1+\frac{1}{2n}}\mathcal{S}_n^+$ and $\Phi_n^- = (\mathcal{S}_n^-)^{-1} \mathcal{F}_{1-\frac{1}{2n}} \mathcal{S}_n^-$. This gives the commutative diagrams shown in Figure [\ref{fig:1}].
 
 It is straightforward to show that $\mathcal{S}_n^{\pm}$ are isometries so the fact that $\Phi_n$ is an isometry is a consequence of $\mathcal{F}_{\nu}$ being an isometry. Instead of simply using this fact from the outset, we chose to supply new proofs as the literature for $\mathcal{F}_{\nu}$ when $-1 < \nu < -\frac{1}{2}$ is quite sparce. While $\Phi_n$ is closely related to the Fourier-Bessel transform and many properties of $\Phi_n$ can be gleaned from the Fourier-Bessel transform, they are inherently different. As far as the authors are aware, while there are extensions of the Fourier-Bessel transform to the whole real line (cf. \cite{rosler}), there are no analogous generalizations of the Fourier-Bessel transform to the whole real line that are similar to $\Phi_n$.
 \vspace{-5mm}
 \begin{center}
 \begin{figure}[htp]
 \begin{tikzpicture}[>=stealth]
  \node (a) at (0,0) {$L^2(\R^+,d\mu_{-1+\frac{1}{2n}})$};
  \node (b) at (0,2) {$L^2(\R^+,d\lambda)$};
  \node (c) at (5,2) {$L^2(\R^+,d\lambda)$};
  \node (d) at (5,0) {$L^2(\R^+,d\mu_{-1+\frac{1}{2n}})$};
  \draw[->] (b) to node [left] {$\mathcal{S}_n^+$} (a);
  \draw[->] (b) to node [above] {$\Phi_n^+$} (c);
  \draw[->] (c) to node [left] {$\mathcal{S}_n^+$} (d);
  \draw[->] (a) to node [above] {$\mathcal{F}_{-1+\frac{1}{2n}}$} (d);
 \end{tikzpicture}
 
 \hfill
 
 \begin{tikzpicture}[>=stealth]
  \node (a) at (0,0) {$L^2(\R^+,d\mu_{1-\frac{1}{2n}})$};
  \node (b) at (0,2) {$L^2(\R^+,d\lambda)$};
  \node (c) at (5,2) {$L^2(\R^+,d\lambda)$};
  \node (d) at (5,0) {$L^2(\R^+,d\mu_{1-\frac{1}{2n}})$};
  \draw[->] (b) to node [left] {$\mathcal{S}_n^-$} (a);
  \draw[->] (b) to node [above] {$\Phi_n^-$} (c);
  \draw[->] (c) to node [left] {$\mathcal{S}_n^-$} (d);
  \draw[->] (a) to node [above] {$\mathcal{F}_{1-\frac{1}{2n}}$} (d);
 \end{tikzpicture} 
 \caption{Commutative diagrams showing the relationships between $\Phi_n^+$ and $\Phi_n^-$ and the Fourier-Bessel transform.} \label{fig:1}
 \end{figure}
\end{center}
\vspace{-5mm}
\section{The Short-Time $\Phi_n$ Transform}

As a result of the linearity and exponential nature of the Fourier kernel, the Fourier transform of a translate of a function $f$ differs from the Fourier transform of $f$ by a modulation. There is unfortunately no similar relationship between the $\Phi_n$ transform of a function $f$ and a translate of $f$. The lack of translation invariance is not a severe drawback as many integral transforms in practice do not have this, e.g. the Fourier-Bessel and Mellin transforms. Consequently, the most natural setting for the $\Phi_n$ transform is in fact as a short-time transform. Recall that the short-time Fourier transform (STFT) \cite{grochenig} of a function $f\in \mathcal{S}(\R)$ with a window $g\in \mathcal{S}(\R)$ is given by
\begin{equation}
 \mathcal{V}_gf(\omega,t) = (2\pi)^{-1/2}\int_{-\infty}^{\infty}e^{-i\omega t'}g(t'-t)f(t')\,dt'. \label{eq:STFT}
\end{equation}

Employing the notation $f_t(t') = f(t'-t)$, this can be rewritten in a more tangible form: $\mathcal{V}_gf(\omega,t) = \mathcal{F}(g_tf)(\omega)$. \eqref{eq:STFT} can instead be written as $\mathcal{V}_gf(\omega, t) = e^{-i\omega t}\mathcal{F}(gf_{-t})(\omega)$, which can be interpreted as the Fourier kernel being centered with the window up to a phase factor. The second realization of the STFT will expedite the development of the short-time $\Phi_n$ transform.

Due to the translational invariance (up to a phase factor) of the Fourier transform, the window need not be centered with the kernel in the definition of the STFT since the power spectra for the two different formulations of the STFT given above are equivalent and thus carry the same information. However since the kernels for $n>1$ are no longer translation invariant, some ambiguity arises when considering short-time analogues of $\Phi_n$. We could consider two different definitions of the short-time $\Phi_n$ transform for a sufficiently nice windowing function $g$ and function $f$:
\begin{gather}
 \mathcal{V}_g^{(n)}f(\omega,t) = \int_{-\infty}^{\infty} \varphi_n(\omega t')g(t'-t)f(t')\,dt', \label{eq:wrongSTPhin} \\
 \mathcal{V}_g^{(n)}f(\omega,t) = \int_{-\infty}^{\infty} \varphi_n(\omega(t'-t))g(t'-t)f(t')\,dt'. \label{eq:STPhin}
\end{gather}

The former clearly resembles the STFT as given in \eqref{eq:STFT}, with $\varphi$ and $f$ centered at $t=0$ and the window $g$, centered at $t$, passing over both. Despite their very different natures, the two notions are in fact equivalent up to an interchange of $g$ and $f$ and a reflection in the time-frequency plane. However the latter definition is more desirable than the former: the short-time $\Phi_n$ transforms of $f$ and a translate of $f$ as given by \eqref{eq:STPhin} differ only by a translation in the time-frequency plane; this is not true with the realization in \eqref{eq:wrongSTPhin}.

Thus we choose to break with the established literature of simply sliding the window across the kernel and function and instead choose to center the kernel with the window $g$ and slide them across the function. That is, we choose the convention given in $\eqref{eq:STPhin}$. We now give the formal definition of the short-time $\Phi_n$ transform and prove two theorems regarding the short-time $\Phi_n$ transform: the reconstruction property and an orthogonality relation.

\begin{defn}
 Let $\omega,t \in\R$ and $g,f\in L^2(\R)$ such that $gh_{-t}\in L^2(\R)$ for all $t$. We define the short-time $\Phi_n$ transform of $f$ with window $g$ to be
\begin{equation}
 \mathcal{V}_g^{(n)}f(\omega,t) = \Phi_n(gf_{-t})(\omega).
\end{equation}
\end{defn}
If $f$ and $g$ are arbitrary functions in $L^2(\R)$, $\Phi_n(gf_{-t})$ may not exist since $gf_{-t}$ in general need not be in $L^2(\R)$, thus the prescription that $gf_{-t}\in L^2(\R)$ is necessary. This restriction is not very strong as it holds for all $f,g\in\mathcal{S}(\R)$ which is a dense subspace of $L^2(\R)$, but for the sake of mathematical rigor, we keep it. Assuming $\Phi_n(gf_{-t})$ exists in the original sense as an integral transform, e.g. if $f$ and $g$ are $n$-Gaussians, then the definition would be exactly as in \eqref{eq:STPhin}. Instead of restricting to functions on which $\Phi_n$ is defined naturally as an integral transform and then extending the results via density arguments, we prefer to work in full generality from the outset for simplicity of argument. With this definition, we may immediately state the theorem.

\begin{thm}
 Let $f,g\in L^2(\R)$ such that $gf_{-t}\in L^2(\R)$, then $f$ may be reconstructed from $\mathcal{V}_g^{(n)}f$ by the following
\begin{equation}
 f(t) = \frac{1}{\langle g,g\rangle}\int_{-\infty}^{\infty}\overline{g(t-\tau)}\Phi_n\mathcal{V}_g^{(n)}f(-t+\tau,\tau)\,d\tau,
\end{equation}

\noindent where $\Phi_n\mathcal{V}_g^{(n)}f$ is understood to be $\Phi_n$ acting on $h_{\tau}(\omega) = \mathcal{V}_g^{(n)}f(\omega,\tau)$, i.e. $\tau$ is constant.
\end{thm}
\begin{proof}
 We first consider the operation of $\Phi_n$ on $\mathcal{V}_g^{(n)}f$. This gives
\begin{equation*}
 \Phi_n\mathcal{V}_g^{(n)}f(-t+\tau,\tau) = \Phi_n(\Phi_n(gf_{-\tau})(\cdot))(-t+\tau) = \Phi_n^2(gf_{-\tau})(-t+\tau).
\end{equation*}
\noindent With the appearance of $\Phi_n^2$, it is natural to break $gf_{-\tau}$ into even and odd parts in order to make use of the fact that $\Phi_n^2$ acts as the identity on even functions and the negative identity on odd functions. We write $f_{-\tau} = f_{-\tau}^+ + f_{-\tau}^-$ and $g = g^+ + g^-$. Therefore it follows that
\begin{eqnarray*}
 \Phi_n\mathcal{V}_g^{(n)}f(-t+\tau,\tau) &=& \Phi_n^2((g^+ + g^-)(f_{-\tau}^+ + f_{-\tau}^-))(-t+\tau) \\
 &=& ((g^+-g^-)(f_{-\tau}^+-f_{-\tau}^-))(-t+\tau) \\
 &=& g(t-\tau)f_{-\tau}(t-\tau) \\
 &=& g(t-\tau)f(t).
\end{eqnarray*}
\noindent Then by above,
\begin{equation*}
 \frac{1}{\langle g,g\rangle}\int_{-\infty}^{\infty}\overline{g(t-\tau)}\Phi_n\mathcal{V}_g^{(n)}f(-t,\tau)\,d\tau = \frac{1}{\langle g,g\rangle}\int_{-\infty}^{\infty} \overline{g(t-\tau)}g(t-\tau)f(t)\,d\tau = f(t).
\end{equation*}
\noindent Thus the theorem is proved.
\end{proof}

With the ability to reconstruct a signal from its short-time $\Phi_n$ transform, it is natural to ask if energy is also preserved as is the case with the STFT. It so happens that an orthogonality relation holds regarding short-time $\Phi_n$ transforms---much like in the case of the STFT \cite{grochenig}---which immediately leads to energy preservation. We shall now state the theorem.

\begin{thm}
 Let $f,\tilde{f},g,\tilde{g}\in L^2(\R)$ such that $gf_{-t},\tilde{g}\tilde{f}_{-t}\in L^2(\R)$, then the following orthogonality relation holds
\begin{equation}
 \int_{\R^2}\mathcal{V}_g^{(n)}f(\omega,t)\overline{\mathcal{V}_{\tilde{g}}^{(n)}\tilde{f}(\omega,t)}\,d\omega \,dt = \langle f,\tilde{f}\rangle\langle g,\tilde{g}\rangle.
\end{equation}
\end{thm}

\begin{proof}
 From the definition of the short-time $\Phi_n$ transform, we have
\begin{eqnarray*}
 \int_{\R^2}\mathcal{V}_g^{(n)}f(\omega,t)\overline{\mathcal{V}_{\tilde{g}}^{(n)}\tilde{f}(\omega,t)}\,d\omega \,dt &=& \int_{-\infty}^{\infty}\int_{-\infty}^{\infty}\Phi_n(gf_{-t})(\omega)\overline{\Phi_n(\tilde{g} \tilde{f}_{-t})(\omega)}\,d\omega \,dt \\
&=& \int_{-\infty}^{\infty}\langle\Phi_n(gf_{-t}),\Phi_n(\tilde{g}{\tilde{f}}_{-t})\rangle_{\omega}\,dt,
\end{eqnarray*}
\noindent where the notation $\langle \cdot, \cdot\rangle_{\omega}$ is an inner product over $\omega$ (with $t$ fixed). Making use the unitarity of $\Phi_n$, this becomes
\begin{eqnarray*}
 \int_{\R^2}\mathcal{V}_g^{(n)}f(\omega,t)\overline{\mathcal{V}_{\tilde{g}}^{(n)}\tilde{f}(\omega,t)}\,d\omega\,dt &=& \int_{-\infty}^{\infty}\langle gf_{-t},\tilde{g}\tilde{f}_{-t}\rangle_{\omega}\,dt \\
&=& \int_{-\infty}^{\infty}\int_{-\infty}^{\infty}g(\omega)\overline{\tilde{g}(\omega)}f_{-t}(\omega)\overline{\tilde{f}_{-t}(\omega)}\,d\omega\, \,dt \\
&=& \int_{-\infty}^{\infty}g(\omega)\overline{\tilde{g}(\omega)}\int_{-\infty}^{\infty}f(\omega+t)\overline{\tilde{f}(\omega+t)}\,dt\,d\omega \\
&=& \langle f,\tilde{f}\rangle\langle g,\tilde{g}\rangle.
\end{eqnarray*}
\noindent Here we have employed Fubini's theorem. Taking $f = \tilde{f}$, $g = \tilde{g}$ and $\langle g,g\rangle = 1$, we see that $\|\mathcal{V}_g^{(n)}f\|_{L^2(\R^2,dt)}^2 = \|f\|_{L^2(\R,dt)}^2$ so the short-time $\Phi_n$ transform preserves energy.
\end{proof}

{\bf Acknowledgments.} The authors  C.L.W., B.G.B. and D.J.K. gratefully acknowledge partial support of this research by grants from Total E\&P USA and PGS. B.G.B. was supported in part by NSF grant DMS-1109545 and DMS-1412524. C.L.W. and D.J.K. acknowledge partial support of this research under Grant E-0608 from the Robert A. Welch Foundation.

\bibliographystyle{plain}
\bibliography{phi_transform_ref}

\end{document}